\documentclass[12pt, a4paper]{article}
\usepackage{amsmath}   %美国数学会的数学公式宏包
\usepackage{comment}
\usepackage{amsfonts}
\usepackage{amssymb}   %美国数学会的数学符号宏包
\usepackage{epsfig}
\usepackage{graphicx}
\usepackage{amsthm}
\usepackage{newlfont}
\usepackage{eucal}
\usepackage{eufrak}
\usepackage{color}


\newtheorem{thm}{Theorem}[section]
\newtheorem{lem}[thm]{Lemma}
\newtheorem{cor}[thm]{Corollary}
\newtheorem{pro}[thm]{Proposition}
\newtheorem{con}[thm]{Conjecture}
\newtheorem{obs}[thm]{Observation}
\theoremstyle{definition}

\newtheorem{exa}[thm]{Example}

\newtheorem{prob}[thm]{Problem}
\theoremstyle{remark}

\graphicspath{{figures/}}
\begin{document}
\title{A relation between the cube polynomials of partial cubes and the clique polynomials of their crossing graphs}

\footnotetext[1]{The work is partially supported by the National Natural Science Foundation of China (Grant No. 12071194, 11571155, 11961067).}

\author{Yan-Ting Xie$^1$, ~Yong-De Feng$^{2}$, ~Shou-Jun Xu$^1, \thanks{Corresponding author. E-mail address: ~shjxu@lzu.edu.cn (S.-J. Xu).}$}

\date{\small $^1$ School of Mathematics and Statistics, Gansu Center for Applied Mathematics, \\Lanzhou University,
Lanzhou, Gansu 730000, China\\
$^2$ College of Mathematics and Systems Science, Xinjiang University, Urumqi,\\ Xinjiang 830046, P.R. China}

\maketitle
\begin{abstract}
Partial cubes are the graphs which can be embedded into hypercubes. The {\em cube polynomial} of a graph $G$ is a counting polynomial of induced hypercubes of $G$, which is defined as $C(G,x):=\sum_{i\geqslant 0}\alpha_i(G)x^i$, where $\alpha_i(G)$ is the number of induced $i$-cubes (hypercubes of dimension $i$) of $G$. The {\em clique polynomial} of $G$ is defined as $Cl(G,x):=\sum_{i\geqslant 0}a_i(G)x^i$, where $a_i(G)$ ($i\geqslant 1$) is the number of $i$-cliques in $G$ and $a_0(G)=1$. Equivalently, $Cl(G, x)$ is exactly the independence polynomial of the complement $\overline{G}$ of $G$. The {\em crossing graph} $G^{\#}$ of a partial cube $G$ is the graph whose vertices are corresponding to the $\Theta$-classes of $G$, and two $\Theta$-classes are adjacent in $G^{\#}$ if and only if they cross in $G$. In the present paper, we prove that for a partial cube $G$, $C(G,x)\leqslant Cl(G^{\#}, x+1)$ and the equality holds if and only if $G$ is a median graph. Since every graph can be represented as the crossing graph of a median graph [SIAM J. Discrete Math., 15 (2002) 235--251], the above necessary-and-sufficient result shows that the study on the cube polynomials of median graphs can be transformed to the one on the clique polynomials of general graphs (equivalently, on the independence polynomials of their complements). In addition, we disprove the conjecture that the cube polynomials of median graphs are unimodal.

\setlength{\baselineskip}{17pt}
{} \vskip 0.1in \noindent%
\end{abstract}
\section{Introduction}
In this paper all graphs we consider are undirected, finite and simple. A {\em hypercube of dimension $n$} (or {\em $n$-cube} for short), denoted by $Q_n$, is a graph whose vertex set is corresponding to the set of 0-1 sequences $x_1x_2\cdots x_n$ with $x_i\in \{0,1\}$, $i=1, 2, \cdots, n$. Two vertices are adjacent if the corresponding 0-1 sequences differ in exactly one digit. $G$ is called {\em partial cube} if it is isomorphic to an isometric subgraph of $Q_n$ for some $n$.

It is known that a relation $\Theta$ on the edge set, called the Djokovi\'c-Winkler relation (see \cite{dj73,w84}), plays an important role in studying the partial cubes. This relation was used by Winkler \cite{w84} to characterize the partial cubes as those bipartite graphs for which $\Theta$ is an equivalence relation on edges. Its equivalence classes are called {\em $\Theta$-classes}. Let $G$ be a partial cube, $\theta_1,\theta_2$ two $\Theta$-classes. We say $\theta_1$ and $\theta_2$ {\em cross} in $G$ if $\theta_2$ occurs in both the components of $G-\theta_1$. The {\em crossing graph} $G^{\#}$ of $G$ is the graph whose vertices are corresponding to the $\Theta$-classes of $G$, and $\theta_1$ and $\theta_2$ are adjacent in $G^{\#}$ if and only if they cross in $G$. The crossing graph was introduced by Bandelt and Dress \cite{bd92} under the name of incompatibility graph and extensively studied by Klav\v zar and Mulder \cite{km02}.

%A graph $G$ is called a {\em median graph} if, for every three vertices $u, v, w$ of $G$, there exists a unique vertex, called the {\em median} of $u,v,w$, that lies on the shortest paths between each pair of $u,v,w$ simultaneously. 

Median graphs are an important subclass of partial cubes, which have many applications in such diverse areas as evolutionary theory, chemistry, literary history, location theory, consensus theory, and computer science. For the structural properties of median graphs, we refer readers to the works cited in references \cite{bv87,km99,mu78,mu80a,mu80b,mu90,mu11}, as well as Chapter 12 of the book \cite{hik11}.

%Graph polynomials are polynomials based on the graph invariants, which can be used for studying some associative properties of graphs. Some common graph polynomials, such as characteristic polynomials, chromatic polynomials, matching polynomials, independence polynomials, et al., have been widely studied.

To study some properties of median graphs, Bre\v sar et al. \cite{bks03} introduced a counting polynomial of hypercubes of a graph $G$, called the {\em cube polynomial},  as follows:
\begin{equation*}
C(G,x):=\sum_{i\geqslant 0}\alpha_i(G)x^i,
\end{equation*}
where $\alpha_i(G)$ is the number of induced $i$-cubes of $G$. Specifically,  $\alpha_0(G)$ denotes the number of vertices and $\alpha_1(G)$ denotes the number of edges in $G$.

An {\em $i$-clique} of a graph $G$ is a complete subgraph with $i$ vertices. Let's define $a_i(G)$ as the number of $i$-clique in $G$ for $i\geqslant 1$ and $a_0(G)=1$. The {\em clique polynomial} of a graph $G$, introduced by Hoede and Li \cite{hl94},  is defined as follows:
\begin{equation*}
Cl(G,x):=\sum_{i\geqslant 0}a_i(G)x^i.
\end{equation*}

Let $P(x)=\sum\limits_{i=0}^mp_ix^i$ and $Q(x)=\sum\limits_{i=0}^nq_ix^i$ be two polynomials with nonnegative coefficients. We say $P(x)\leqslant Q(x)$ if $m\leqslant n$ and $p_i\leqslant q_i$ for $0\leqslant i\leqslant m$. If $P(x)\leqslant Q(x)$ and $P(x)\neq Q(x)$, we say $P(x)<Q(x)$.

Let $G$ be a median graph. Zhang et al. \cite{zss13} proved that $C(G,x)=\sum\limits_{i=0}^mb_i(G)(x+1)^i$ where $b_0(G)=1$ and $b_i(G)$ is a positive integer for each $i$ with $1\leqslant i\leqslant m$. They further provided an expression of $b_i(G)$. % And further, they gave a combinatorial meaning of $b_i(G)$. 
In the present paper, we reveal a combinatorial meaning of $b_i(G)$ as the number of $i$-cliques of $G^{\#}$, i.e., $b_i(G)=a_i(G^{\#})$. Moreover, we establish the following relationship between the cube polynomials of partial cubes and the clique polynomials of their crossing graphs.

\begin{thm}\label{thm:CubePandCliqueP}
Let $G$ be a partial cube and $G\neq K_1$. Then
\begin{equation}\label{eq:CubePandCliqueP}
C(G,x)\leqslant Cl(G^{\#},x+1)
\end{equation}
and the equality holds if and only if $G$ is a median graph.
\end{thm}

For a general graph $G$, the {\em simplex graph} $S(G)$ of $G$ is defined as the graph whose vertices are the cliques of $G$ (including the empty graph), two vertices being adjacent if, as cliques of $G$, they differ in exactly one vertex (see \cite{bv91,km02}). The simplex graph $S(G)$ of $G$ is a median graph. About the crossing graphs of median graphs, Klav\v zar and Mulder derived the following theorem.

\begin{thm}{\em\cite{km02}}\label{thm:EveryGraphisaCrossingGraph}
Every graph can be represented as the crossing graph of some median graph. Specifically, for any graph $G$, we have $G=S(G)^{\#}$.
\end{thm}

Combined with the fact that the independence polynomial of a graph is equal to the clique polynomial of its complement, Theorem \ref{thm:CubePandCliqueP} indicates  that the investigation of the cube polynomials of median graphs can be transformed to the study on the clique polynomials and the independence polynomials of general graphs.

A sequence $(s_1,s_2,\cdots,s_n)$ of nonnegative real numbers is {\em unimodal} if
\begin{equation*}
s_1\leqslant s_2\leqslant\cdots\leqslant s_m\geqslant\cdots\geqslant s_{n-1}\geqslant s_n
\end{equation*}
for some integer $1\leqslant m\leqslant n$ and {\em log-concave} if 
\begin{equation*}
s_{i-1}s_{i+1}\leqslant s^2_i,\qquad\mbox{for }2\leqslant i\leqslant n-1.
\end{equation*}
The sequence is said to {\em have no internal zeros} if there are not three indices $i < j < k$ such that $s_i,s_k>0$ and $s_j=0$. In particular, the positive sequences have no internal zeros. It is well-known that a log-concave sequence with no internal zeros is unimodal \cite{b89}. A polynomial is called {\em unimodal} (resp. {\em log-concave}) if the sequence of its coefficients is unimodal (resp. log-concave). By the definitions, the coefficient sequences of the cube polynomials, the clique polynomials and the independence polynomials of graphs are positive. Thus, for these graph polynomials, the log-concavity is stronger than the unimodality. 

When studying graph polynomials, the unimodality and the log-concavity are always considered. For instance, it has been proved that the matching polynomials of graphs \cite{hl72}, the independence polynomials of claw-free graphs \cite{h90} and the signless chromatic polynomials of graphs \cite{h12} are log-concave, but the conjecture about the unimodality of independence polynomials of trees is still open \cite{amse87}.

Zhang et al. %proved that the Clar covering polynomial of a Kekul\'ean hexagonal system is equal to the cube polynomial of its resonance graph, where the resonance graph of a Kekul\'ean hexagonal system is a median graph (see \cite{zls08}). Furthermore, the authors 
conjectured:

\begin{con}{\em\cite{zss13}}\label{con:CubePsareUnimodal}
The cube polynomials of median graphs are unimodal.
\end{con}

We disprove this conjecture by providing counterexamples, which are obtained from $Q_n$ ($n\geqslant 9$) by attaching sufficiently many pendant vertices.

The paper is organized as follows. In the next section, we introduce some terminology and properties of partial cubes and median graphs. Then, we prove the main theorem (i.e., Theorem \ref{thm:CubePandCliqueP}) in Section 3 and disprove Conjecture \ref{con:CubePsareUnimodal} in Section 4. Finally, we conclude the paper and propose some future problems in Section 5.

\section{Preliminaries}
Let $G$ be a graph with vertex set $V(G)$ and edge set $E(G)$. For $S\subseteq V(G)$, the subgraph induced by $S$ is denoted by $G[S]$. For $v\in V(G)$, $N_G(v)$ is the neighbourhood of $v$ in $G$, i.e., $N_G(v):=\{u\in V(G)|uv\in E(G)\}$. The subgraph $G[N_G(v)]$ is written as $G_v$ simply. For $u, v\in V(G)$, the {\em distance} $d_{G}(u,v)$ (we will drop the subscript $G$ if no confusion can occur) is the length of the shortest path between $u$ and $v$ in $G$.  We call a shortest path from $u$ to $v$ a $u,v$-{\em geodesic}. A subgraph $H$ of $G$ is called {\em isometric} if for any $u,v\in V(H)$, $d_H(u,v)=d_G(u,v)$, and further, if for any $u,v\in V(H)$, all $u,v$-geodesics are contained in $H$, we call $H$ a {\em convex} subgraph of $G$. Obviously, the convex subgraphs are isometric and the isometric ones are induced and connected. A graph $G$ is called  a {\em partial cube} if it is isomorphic to an isometric subgraph of $Q_n$ for some $n$.

%A {\em hypercube of dimension $n$} (or {\em $n$-cube} for short), denoted by $Q_n$, is a graph whose vertex set is corresponding to the set of 0-1 sequences $x_1x_2\cdots x_n$ with $x_i\in \{0,1\}$, $i=1, 2, \cdots, n$. Two vertices are adjacent if the corresponding 0-1 sequences differ in exactly one digit.  %, or in other words, it can be embedded in $Q_n$ isometrically.

The {\em Djokovi\'c-Winkler relation} (see \cite{dj73,w84}) $\Theta_{G}$ is a binary relation on $E(G)$ defined as follows: Let $e=uv$ and $f=xy$ be two edges in $G$, $e\,\Theta_{G}\,f\iff d_{G}(u,x)+d_{G}(v,y)\neq d_{G}(u,y)+d_{G}(v,x)$. If $G$ is bipartite, there is another equivalent definition of the Djokovi\'c-Winkler relation: $e\,\Theta_{G}\,f\iff d_{G}(u,x)=d_{G}(v,y)$ and $d_{G}(u,y)=d_{G}(v,x)$. Winkler \cite{w84} proved that a graph $G$ is a partial cube if and only if $G$ is bipartite and $\Theta_{G}$ is an equivalence relation on $E(G)$. The following property can be obtained easily by definition.

\begin{obs}\label{obs:ThetaRestrictonSubgraph}
Let $G$, $G'$, $G''$ be three graphs. If $G$ is an isometric subgraph of $G'$ and $G'$ is an isometric subgraph of $G''$, then $G$ is an isometric subgraph of $G''$. In prticular, if $G'$ is a partial cube (in this case, $G''=Q_n$ for some integer $n$), then $G$ is also a partial cube. Moreover, for any $e,f\in E(G)$, $e\,\Theta_{G}\,f\iff e\,\Theta_{G'}\,f$.
\end{obs}

Let $G$ be a partial cube. We call an equivalence class on $E(G)$ a {\em $\Theta$-class}. The {\em isometric dimension} of $G$, denoted by $\mathrm{idim}(G)$, is the smallest integer $n$ satisfying that $G$ is the isometric subgraph of $Q_n$, which coincides with the number of $\Theta$-classes \cite{dj73}. For $e=uv\in E(G)$, we denote the $\Theta$-class containing $uv$ as $F^G_{uv}$, i.e., $F^G_{uv}:=\{f\in E(G)|f\,\Theta_{G}\,e\}$. If we don't focus on which edges is contained in, we can also denote the $\Theta$-class by $\theta_1,\theta_2,\cdots$. Moreover, we denote $W^G_{uv}:=\{w\in V(G)|d_{G}(u,w)<d_{G}(v,w)\}$ and $U^G_{uv}:=\{w\in V(G)|w\in W^G_{uv}\mbox{ and }w\mbox{ is incident with an edge in }F^G_{uv}\}$. Except for Subsection 3.2, we will drop the subscript of $\Theta_G$ and the superscript of $F^G_{uv}$, $W^G_{uv}$ and $U^G_{uv}$ in the following. 

The following property is obvious.

\begin{obs}\label{obs:FInduceIsomorphism}
Let $G$ be a partial cube and $uv$ is an edge in $E(G)$. $F_{uv}$ induces an isomorphism between the induced subgraphs $G[U_{uv}]$ and $G[U_{vu}]$. Further, if $u_1u_2$, $v_1v_2$ are corresponding edges in $G[U_{uv}]$ and $G[U_{vu}]$ respectively, then $u_1u_2\,\Theta\,v_1v_2$.
\end{obs}

About $W_{uv}$, Djokovi\'c obtained the following property:

\begin{pro}{\em\cite{dj73}}\label{pro:WisConvex}
Let $G$ be a partial cube. Then $G[W_{uv}]$ and $G[W_{vu}]$ are convex in $G$ for any $uv\in E(G)$.
\end{pro}

\begin{comment}
About the geodesics in partial cubes, the following proposition is useful.

\begin{pro}{\em\cite{ik00}}\label{pro:DWRelationonGeodesic}
Let $G$ be a partial cube with a path $P$ in it. $P$ is a geodesic in $G$ if and only if no two distinct edges on it are in relation $\Theta$. In particular, two adjacent edges in a partial cube can't be in relation $\Theta$.
\end{pro}
\end{comment}

Let $H$ be an induced subgraph of $G$. Denote $\partial\, H=\{uv\in E(G)|u\in V(H),v\not\in V(H)\}$. About the convex subgraphs of bipartite graphs, Imrich and Klav\v zar obtained the following proposition:

\begin{pro}{\em\cite{ik98}}\label{pro:ConvexityLemma}
An induced connected subgraph $H$ of a bipartite graph $G$ is convex if and only if no edge of $\partial\, H$ is in relation $\Theta$ to an edge in $H$.
\end{pro}

In particular, we have

\begin{obs}\label{obs:ConvexSubgraphofHypercube}
Every convex subgraph of the hypercube $Q_n$ is a hypercube $Q_r$ for some integer $r\leqslant n$.
\end{obs}

$G$ is called a {\em median graph} if for every three different vertices $u,v,w\in V(G)$, there exists exactly one vertex $x\in V(G)$ (maybe $x\in\{u,v,w\}$), called the {\em median} of $u,v,w$, satisfying that $d(u,x)+d(x,v)=d(u,v)$, $d(u,x)+d(x,w)=d(u,w)$ and $d(v,x)+d(x,w)=d(v,w)$, that is, there exist the geodesics between each pair of $u,v,w$ where $x$ lies on all of them. If $H$ is a convex subgraph of a median graph $G$, then $H$ is also a median graph. There are many equivalent characterizations of median graphs (see \cite{km99}). The most famous one is the following proposition:

\begin{pro}{\em\cite{ik00}}\label{pro:UisConvex}
A graph $G$ is a median graph if and only if $G$ is a partial cube and $G[U_{uv}]$ is convex in $G$ for every $uv\in E(G)$.
\end{pro}

Let $G$ be a graph and $G_1,G_2$ two isometric subgraphs with $G=G_1\cup G_2$ and $G_0=G_1\cap G_2$ not empty, where there are no edges between $G_1- G_2$ and $G_2- G_1$. Let $G^*_i$ be isomorphic copy of $G_i$ for $i=1,2$. For every $u\in V(G_0)$, let $u_i$ be the corresponding vertex in $G^*_i$ ($i=1,2$). The {\em expansion} $G^*$ of $G$ with respect to $\{G_1,G_2\}$ is the graph obtained from the disjoint union $G^*_1$ and $G^*_2$ by adding an edge between the corresponding vertices $u_1$ and $u_2$ for each vertex $u\in G_0$. It is known that partial cubes are characterized as graphs that can be obtained from $K_1$ by a sequence of expansions \cite{c88}. If $G$ is a partial cube, all edges $u_1u_2$ for $u\in G_0$ compose a new $\Theta$-class by the definition of the Djokovi\'c-Winkler relation. So $\mathrm{idim}(G^*)=\mathrm{idim}(G)+1$. The expansion is {\em convex} if $G_0$ is convex, and {\em peripheral} if $G_0=G_1$ or $G_0=G_2$. The second equivalent characterization of median graphs we will use is:

\begin{pro}{\em\cite{mu90}}\label{pro:ConvexPeripheralExpansions}
Let $G$ be a connected graph. $G$ is a median graph if and only if it can be obtained from $K_1$ by a sequence of peripheral convex expansions.
\end{pro}

Let $G$ be a partial cube and not $K_1$, $F_{ab},F_{uv}$ two $\Theta$-classes of $G$. We say $F_{ab},F_{uv}$ {\em cross} if $W_{ab}\cap W_{uv}\neq\emptyset$, $W_{ba}\cap W_{uv}\neq\emptyset$, $W_{ab}\cap W_{vu}\neq\emptyset$ and $W_{ba}\cap W_{vu}\neq\emptyset$. For a subgraph $H$ of $G$, we say $F_{uv}$ {\em occurs} in $H$ if there is an edge of $F_{uv}$ in $E(H)$. Another equivalent definition of crossing is: $F_{ab}$ and $F_{uv}$ cross if $F_{ab}$ occurs in both $G[W_{uv}]$ and $G[W_{vu}]$. The {\em crossing graph} of $G$ (see \cite{km02}), denoted by $G^{\#}$, is the graph whose vertices are corresponding to the $\Theta$-classes of $G$, and $\theta_1=F_{ab},\theta_2=F_{uv}$ are adjacent in $G^{\#}$ if and only if they cross in $G$. The following proposition is an equivalent expression of crossing relation.

\begin{pro}{\em\cite{km02}}\label{pro:CrossinIsometricCycle}
Let $G$ be a partial cube, $\theta_1,\theta_2$ two $\Theta$-classes. Then $\theta_1$ and $\theta_2$ cross if and only if they occur on an isometric cycle $C$, i.e., $E(C)\cap\theta_1\neq\emptyset$, $E(C)\cap\theta_2\neq\emptyset$.
\end{pro}

\begin{comment}
The {\em cube polynomial} of a graph $G$, introduced by Bre\v sar et al. \cite{bks03}, is defined as follows:
\begin{equation*}
C(G,x)=\sum_{i\geqslant 0}\alpha_i(G)x^i,
\end{equation*}
where $\alpha_i(G)$ is the number of induced $i$-cubes of $G$.

The {\em clique polynomial} of $G$, introduced by Hoede and Li \cite{hl94}, is defined as follows:
\begin{equation*}
Cl(G,x)=\sum_{i\geqslant 0}a_i(G)x^i.
\end{equation*}
where $a_i(G)$ ($i\geqslant 1$) is the number of $i$-cliques in $G$ and we define $a_0(G)=1$.
\end{comment}

\section{Proof of Theorem \ref{thm:CubePandCliqueP}}

The proof of Theorem \ref{thm:CubePandCliqueP} is organized as follows. First, we prove that the `$=$' in  (\ref{eq:CubePandCliqueP}) holds if $G$ is a median graph. Then, we prove that $C(G,x)<Cl(G^{\#},x+1)$ if $G$ is not a median graph.

\subsection{$C(G,x)=Cl(G^{\#},x+1)$ if $G$ is a median graph}
Klav\v zar and Mulder obtained the following lemma:

\begin{lem}{\em\cite{km02}}\label{lem:WeakCrossin4Cycle}
Let $G$ be a median graph and $uv,uw\in E(G)$. If $F_{uv}$ and $F_{uw}$ cross, then $v,u,w$ are in a 4-cycle.
\end{lem}

Let $G$ be a partial cube, $\theta_1,\theta_2$ two $\Theta$-classes. A 4-cycle $uvwxu$ is called to be {\em $\theta_1,\theta_2$-alternating} if $uv,xw\in\theta_1$ and $ux,vw\in\theta_2$. We have the following lemma, which is stronger than Lemma \ref{lem:WeakCrossin4Cycle}.

\begin{lem}\label{lem:Crossin4Cycle}
Let $G$ be a median graph, $\theta_1,\theta_2$ two $\Theta$-classes. Then $\theta_1$ and $\theta_2$ cross if and only if there exists a $\theta_1,\theta_2$-alternating 4-cycle.
\end{lem}
\begin{proof}
{\em Sufficiency.} Assume $C=uvwxu$ is a 4-cycle where $uv,xw\in\theta_1$ and $ux,vw\in\theta_2$. Since 4-cycles must be isometric in a bipartite graph, by Proposition \ref{pro:CrossinIsometricCycle}, $\theta_1$ and $\theta_2$ cross.

{\em Necessity.} Denote the two components of $G-\theta_1$ as $G_1,G_2$. By the definition of crossing, there exist $a_1b_1,a_2b_2\in\theta_2$ satisfying that $a_1b_1\in E(G_1)$, $a_2b_2\in E(G_2)$ and $d(a_1,a_2)=d(b_1,b_2)=d(a_1,b_2)-1=d(b_1,a_2)-1$. Let $P$ be an $a_1,a_2$-geodesic. Since $a_1\in V(G_1)$, $a_2\in V(G_2)$ and $\theta_1$ is an edge cutset, there exists an edge in $E(P)\cap\theta_1$, denoted by $c_1c_2$. Since $c_1,c_2$ is on $P$, which is an $a_1,a_2$-geodesic, by Proposition \ref{pro:UisConvex}, $c_1,c_2\in U_{a_1b_1}$. Thus, there exist edges $c_1d_1,c_2d_2\in\theta_2$. By Observation \ref{obs:FInduceIsomorphism}, $d_1,d_2$ are adjacent and $d_1d_2\in\theta_1$, that is, the 4-cycle $c_1c_2d_2d_1c_1$ is $\theta_1,\theta_2$-alternating.
\end{proof}

Let $G$ be a partial cube, $H$ a convex subgraph of $G$ and $\theta_1,\theta_2$ two crossing $\Theta$-classes. If there is a $\theta_1,\theta_2$-alternating 4-cycle in $H$, we say that $\theta_1$ and $\theta_2$ {\em cross in} $H$.

\begin{lem}\label{lem:CrossinConvexSubgraph}
Let $G$ be a median graph, $H$ a convex subgraph of $G$ and $\theta_1,\theta_2$ two crossing $\Theta$-classes. If both $\theta_1$ and $\theta_2$ occur in $H$, then they cross in $H$.
\end{lem}
\begin{proof}
Assume $e=uv$, $f=xw$ are two edges such that $e\in E(H)\cap\theta_1$ and $f\in E(H)\cap\theta_2$. We distinguish two cases to discuss.

\textbf{Case 1.} $e$ and $f$ are adjacent.

W.l.o.g., assume $v=x$. By Lemma \ref{lem:WeakCrossin4Cycle}, $u,v,w$ are in a 4-cycle, say $uvwyu$. Since $u,v,w\in V(H)$, the 4-cycle $uvwyu$ is in $H$ by the convexity of $H$.  Thus, $\theta_1$ and $\theta_2$ cross in $H$.

\textbf{Case 2.} $e$ and $f$ are not adjacent.

W.l.o.g., assume $d(v,x)=d(u,x)-1=d(v,w)-1$. Since $\theta_1$ and $\theta_2$ cross, by Lemma \ref{lem:Crossin4Cycle}, there exists a $\theta_1,\theta_2$-alternating 4-cycle, denoted by $C=abcda$. W.l.o.g., we assume $ab,cd\in\theta_1$, $ad,bc\in\theta_2$, $a\in U_{vu}\cap U_{xw}$, $b\in U_{uv}\cap U_{xw}$, $c\in U_{uv}\cap U_{wx}$ and $d\in U_{vu}\cap U_{wx}$. If $C$ is in $H$, the lemma holds. If some vertices of $C$ are in $H$ and others are not, then there are some edges of $C$ in $\partial\, H$, a contradiction with Proposition \ref{pro:ConvexityLemma}. Now assume $C$ is in $G-H$. By the definition of median graphs, let $v'$ be the median of $a,v,x$. Since $a,v\in U_{vu}$, $a,x\in U_{xw}$ and $v'$ is on both an $a,v$-geodesic and an $a,x$-geodesic, by Proposition \ref{pro:UisConvex}, $v'\in U_{vu}\cap U_{xw}$. Then there exist edges $v'u'$, $v'w'$ such that $v'u'\in \theta_1$ and $v'w'\in \theta_2$ (see Fig. \ref{fig:CrossG0}). Since $v'$ is on a $v,x$-geodesic, $v'\in V(H)$ by the convexity of $H$. Then $u',w'\in V(H)$ by Proposition \ref{pro:ConvexityLemma}. After considering the edges $v'u'$, $v'w'$ which is similar to Case 1, we conclude that $\theta_1$ and $\theta_2$ cross in $H$.
\end{proof}

\begin{figure}[!htbp]
\centering
\scalebox{0.4}[0.4]{\includegraphics{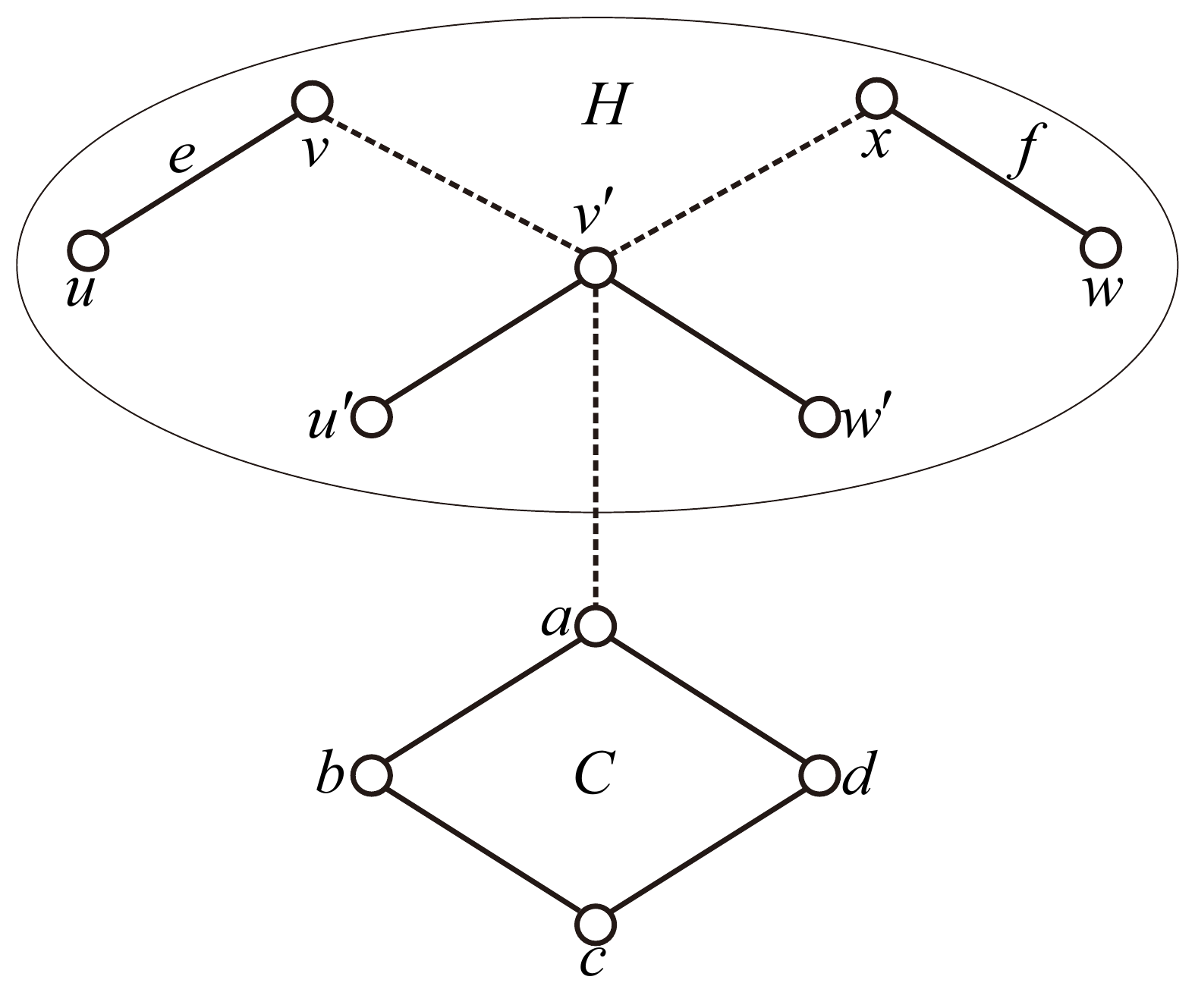}}
\caption{Illustration for Case 2 in proof of Lemma \ref{lem:CrossinConvexSubgraph}}{\label{fig:CrossG0}}
\end{figure}

Now, we consider the crossing graphs. Let $G$ be a partial cube, $G^{\#}$ its crossing graph. If $v$ is a vertex in $G^{\#}$, we denote its corresponding $\Theta$-class in $G$ by $\theta_v$ in what follows. Let $H$ be a convex subgraph of $G$. The crossing graph of $H$ is denoted by $H^{\#}$  (by Observation \ref{obs:ThetaRestrictonSubgraph}, $H^{\#}$ is well-defined since the convex subgraph of a partial cube is still a partial cube). Then, the statement `$v\in V(H^{\#})$' is corresponding to `$\theta_v$ occurs in $H$', and `$uv\in E(H^{\#})$' is corresponding to `$\theta_u$, $\theta_v$ cross in $H$'. By Lemma \ref{lem:CrossinConvexSubgraph}, we can obtain the following lemma easily.

\begin{lem}\label{lem:ConvexandInducedSubgraph}
Let $G$ be a median graph. If $H$ is convex subgraph of $G$, then $H^{\#}$ is an induced subgraph of $G^{\#}$.
\end{lem}

A pair of induced subgraphs $\{G_1,G_2\}$ of $G$ is called a {\em cubical cover} if $G=G_1\cup G_2$ and every induced hypercube in $G$ is contained in at least one of the $G_1$ and $G_2$. Bre\v sar et al. gave a recursive formula of cube polynomials about the expansion with respect to the cubical cover:

\begin{lem}{\em\cite{bks03}}\label{lem:RecursiveFormulaofCubeP}
Let $G$ be a graph constructed by the expansion with respect to the cubical cover $\{G_1,G_2\}$ with $G_0=G_1\cap G_2$. Then
\begin{equation*}
C(G,x)=C(G_1,x)+C(G_2,x)+xC(G_0,x).
\end{equation*}
\end{lem}

If an expansion with respect to $\{G_1,G_2\}$ with $G_0=G_1\cap G_2$ is peripheral, i.e., $G_0=G_2$, then it is obvious that $\{G_1,G_2\}$ is a cubical cover. Thus, by Proposition \ref{pro:ConvexPeripheralExpansions}, for median graphs, we have

\begin{cor}\label{lem:RecursiveFormulaforMedianGraph}
Let $G$ be a median graph constructed by the peripheral convex expansion with respect to $\{G_0,G_1\}$ where $G_0$ is a convex subgraph of $G_1$. Then
\begin{equation}\label{eq:RecursiveFormulaforMedianGraph}
C(G,x)=C(G_1,x)+(x+1)C(G_0,x).
\end{equation}
\end{cor}

Recall that $G_v=G[N_G(v)]$. The recursive formula of clique polynomials is given by the following lemma:

\begin{lem}{\em\cite{hl94}}\label{lem:RecursiveFormulaofCliqueP}
Let $G$ be a graph, $v$ a vertex of $G$. Then
\begin{equation}\label{eq:RecursiveFormulaofCliqueP}
Cl(G,x)=Cl(G-v,x)+xCl(G_v,x).
\end{equation}
\end{lem}

Now, we give the main result of this subsection.

\begin{thm}\label{thm:=forMedianGraph}
Let $G$ be a median graph and $G\neq K_1$. Then
\begin{equation}\label{eq:=forMedianGraph}
C(G,x)=Cl(G^{\#},x+1).
\end{equation}
\end{thm}
\begin{proof}
We prove the equality (\ref{eq:=forMedianGraph}) by induction on $\mathrm{idim}(G)$ (or equivalently, $|V(G^{\#})|$). 

For the base case when $\mathrm{idim}(G)=1$, $G\cong K_2$, $G^{\#}\cong K_1$. In this case, $C(G,x)=x+2$, $Cl(G^{\#},x)=x+1$. So the equality (\ref{eq:=forMedianGraph}) holds.

Now, assume (\ref{eq:=forMedianGraph}) holds for all median graphs with isometric dimension at most $n-1$. Let $G$ be a median graph with $\mathrm{idim}(G)=n$. By Proposition \ref{pro:ConvexPeripheralExpansions}, $G$ can be obtained from a median graph $G_1$ by a peripheral convex expansion with respect to $\{G_0,G_1\}$ where $G_0$ is a convex subgraph of $G_1$. Then, by the definition of median graphs, $G_0$ is a median graph and further $\mathrm{idim}(G_1)=n-1$, $\mathrm{idim}(G_0)\leqslant n-1$. By (\ref{eq:RecursiveFormulaforMedianGraph}) and the induction hypothesis, we have
\begin{equation}\label{eq:InductionHypothesis}
C(G,x)=C(G_1,x)+(x+1)C(G_0,x)=Cl(G_1^{\#},x+1)+(x+1)Cl(G_0^{\#},x+1).
\end{equation}
Let $\theta_v$ be the $\Theta$-class obtained by the peripheral convex expansion. By Lemma \ref{lem:RecursiveFormulaofCliqueP}, we only need to prove that $G_1^{\#}=G^{\#}-v$ and $G_0^{\#}=(G^{\#})_v$ (to avoid confusion, we denote the subgraph induced by the neighbourhood of $v$ in $G^{\#}$ by $(G^{\#})_v$).

Let $ab$ be an edge in $\theta_v$. W.l.o.g., assume $G[W_{ab}]=G[U_{ab}]=G_0$, $G[W_{ba}]=G_1$. We denote $G[U_{ba}]:=G'_0$. Since $G_0$ and $G'_0$ are isomorphic and their corresponding edges are $\Theta$-related by Observation \ref{obs:FInduceIsomorphism}, two $\Theta$-classes $\theta_u$ and $\theta_w$ ($u,w\neq v$) cross in $G_0$ if and only if they cross in $G'_0$, so do in $G_1$. Thus, $u,w$ are adjacent in $G^{\#}$ if and only if they are adjacent in $G_1^{\#}$. Then $G_1^{\#}=G^{\#}-v$. Let's consider $G_0^{\#}$. By the definition of crossing, $\theta_u$ occurs in $G_0$ if and only if it crosses $\theta_v$ in $G$. Thus $V(G_0^{\#})=N_{G^{\#}}(v)$. Since $G_0$ is a convex subgraph of $G$ by Proposition \ref{pro:WisConvex}, we obtain $G_0^{\#}=(G^{\#})_v$ by Lemma \ref{lem:ConvexandInducedSubgraph}.

By (\ref{eq:RecursiveFormulaofCliqueP}), we obtain
\begin{equation*}
Cl(G_1^{\#},x+1)+(x+1)Cl(G_0^{\#},x+1)=Cl(G^{\#}-v,x+1)+(x+1)Cl((G^{\#})_v,x+1)=Cl(G^{\#},x+1).
\end{equation*}
Thus, the induction step completes the proof.
\end{proof}
\subsection{$C(G,x)<Cl(G^{\#},x+1)$ if $G$ is not a median graph}
Let $G$ be a graph. For any two vertices $u,v\in V(G)$, the {\em interval} $I_G(u,v)$ between $u$ and $v$ is a subset of $V(G)$ which is defined as follows: $I_G(u,v):=\{w\in V(G)|w\mbox{ is on a geodesic between }u$ $\mbox{and }v\}$. Let $X$ be a set. The power set $\mathcal{P}(X)$ is the set of subsets of $X$, i.e., $\mathcal{P}(X):=\{Y|Y\subseteq X\}$. Let $S$ be a subset of $V(G)$. Let $\ell_G$ be the self-map of $\mathcal{P}(V(G))$ defined by $\ell_G(S):=\bigcup\limits_{u,v\in S}I_G(u,v)$. Let's denote $\ell^1_G(S):=\ell_G(S)$ and $\ell^i_G(S)=\ell(\ell^{i-1}_G(S))$ for each integer $i\geqslant 2$. The {\em convex hull} of $S$ in $G$ is defined as $co_G(S)=\bigcup\limits_{i\in\mathbb{N}}\ell^i_G(S)$. We can see that $G[co_G(S)]$ is the smallest convex subgraph of $G$ containing $S$. In particular, if $H$ is a subgraph of $G$, we denote  $co_G(H):=G[co_G(V(H))]$ and call it the {\em convex hull of the subgraph $H$} in $G$. Another characterization of median graphs which was proved in an unpublished manuscript due to Bandelt in 1982, and can be found in \cite{km99}, reads as follows:

\begin{thm}{\em\cite{km99}}\label{thm:ConvexHullofIsometricCycle}
Let $G$ be a connected graph. $G$ is a median graph if and only if the convex hull of any isometric cycle of $G$ is a hypercube.
\end{thm}

Let $G$ be a graph. Let $\mathcal{C}(G)$ be the set of all isometric cycles of $G$. We define a binary relation `$\leqslant_{\mathcal{C}(G)}$' on $\mathcal{C}(G)$ as follows: For any $C,C'\in\mathcal{C}(G)$, $C\leqslant_{\mathcal{C}(G)}C'\iff co_G(V(C))\subseteq co_G(V(C'))$. It is easy to see that `$\leqslant_{\mathcal{C}(G)}$' is a partial order on $\mathcal{C}(G)$. We say that $C\in\mathcal{C}(G)$ is {\em maximal} if it is a maximal element in the partial order  `$\leqslant_{\mathcal{C}(G)}$'. The set of maximal isometric cycles in $G$ is denoted by $\mathcal{C}_{\max}(G)$.

Let $G$ be a partial cube with $\mathrm{idim}(G)=n$. Now, we construct $G^{+}$ based on $G$. In what follows, $G$ is considered as an isometric subgraph of $Q_n$. Denote $G^{(0)}:=G$ and $S_0:=V(G)$, and further, for $i\geqslant 0$,
\begin{equation}\label{eq:Si}
S_{i+1}:=S_i\cup\bigcup_{C\in\mathcal{C}_{\max}(G^{(i)})}co_{Q_n}(V(C))
\end{equation}
and
\begin{equation}\label{eq:Gi}
G^{(i+1)}:=Q_n[S_{i+1}]
\end{equation}
recursively. Since $G^{(i)}$ is the subgraph of $Q_n$ for each $i\geqslant 0$, every subgraph of $G^{(i)}$ is also the subgraph of $Q_n$. Thus, every $co_{Q_n}(V(C))$ in (\ref{eq:Si}) is well-defined, so are $S_{i+1}$ and $G^{(i+1)}$. Since $Q_n$ is finite, there must exist a smallest integer $l\geqslant 0$ that all $S_i$'s (resp. $G^{(i)}$'s) are equal for $i\geqslant l$.  We define $G^{+}:=G^{(l)}$. 
\begin{figure}[!htbp]
\centering
\setlength{\unitlength}{1mm}
\begin{picture}(143,71)
\put(0,10){\scalebox{0.5}[0.5]{\includegraphics{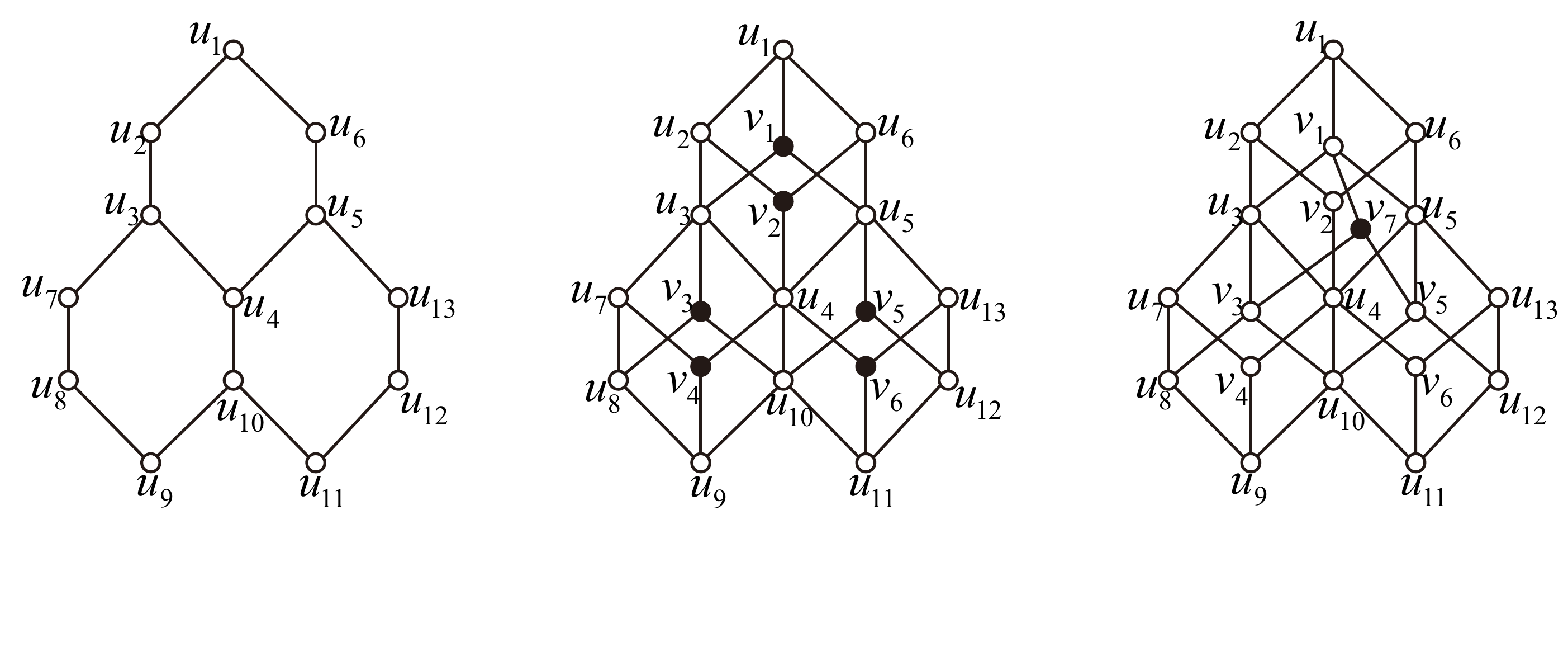}}}
\put(13,12){$G^{(0)}=G$}
\put(19,2){(1)}
\put(50,12){$\begin{array}{l}
G^{(1)}=Q_6[V(G^{(0)})\\
 \cup\{v_1,v_2,v_3,v_4,v_5,v_6\}]
\end{array}$}
\put(68,2){(2)}
\put(100,12){$\begin{array}{l}
G^+=G^{(2)}\\
=Q_6[V(G^{(1)})\cup\{v_7\}]
\end{array} $}
\put(118,2){(3)}
\end{picture}
\caption{An example of the construction of $G^{+}$}{\label{fig:ConstructionG+}}
\end{figure}
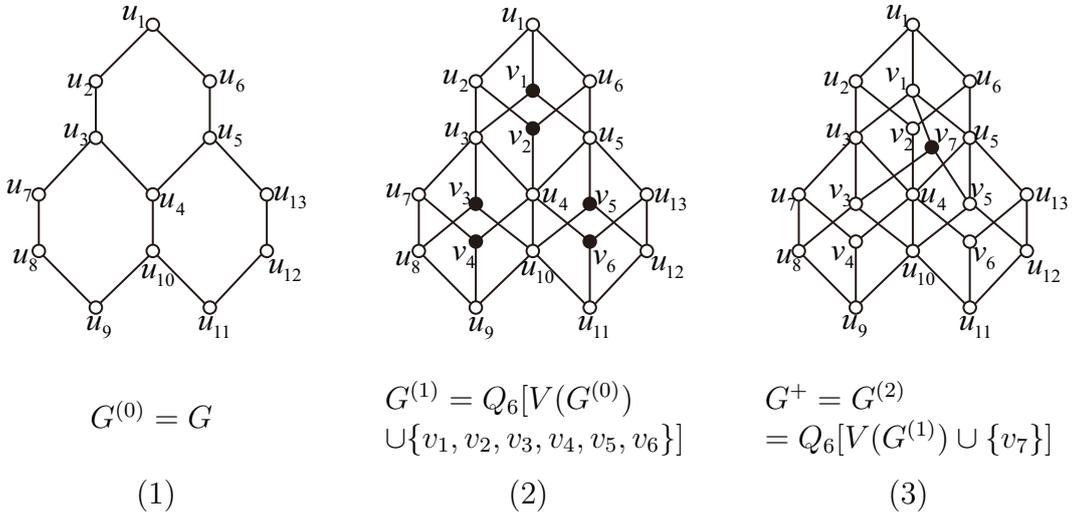

\begin{exa}\label{exa:ConstructionG+}
Let $G=G^{(0)}$ be the partial cube illustrated in Fig. \ref{fig:ConstructionG+} (1), which can be considered as an isometric subgraph of $Q_6$. $\mathcal{C}_{\max}(G^{(0)})=\{C_1,C_2,C_3\}$, where $C_1=u_1u_2u_3u_4u_5u_6u_1$, $C_2=u_3u_7u_8u_9u_{10}u_4u_3$ and $C_3=u_5u_4u_{10}u_{11}u_{12}u_{13}u_5$. Since $C_1,C_2,C_3$ are 6-cycles, $co_{Q_6}(C_i)$ ($i=1,2,3$) are isomorphic to $Q_3$. Assume $co_{Q_6}(V(C_1))=V(C_1)\cup\{v_1,v_2\}$, $co_{Q_6}(V(C_2))=V(C_2)\cup\{v_3,v_4\}$ and $co_{Q_6}(V(C_3))=V(C_3)\cup\{v_5,v_6\}$, then, by the definition, $G^{(1)}=Q_6[V(G^{(0)})\cup\{v_1,v_2,v_3,v_4,v_5,v_6\}]$ (see Fig. \ref{fig:ConstructionG+} (2)). For $G^{(1)}$, $\mathcal{C}_{\max}(G^{(1)})=\{C_1,C_2,C_3,C_4\}$, where $C_4=v_1u_3v_3u_{10}v_5u_5v_1$. $co_{Q_6}(C_4)$ is also isomorphic to $Q_3$ and $u_4\in co_{Q_6}(V(C_4))$. Assume $co_{Q_6}(V(C_4))$ $=V(C_4)\cup\{u_4,v_7\}$. Then $G^{(2)}=Q_6[V(G^{(1)})\cup\{v_7\}]$ (see Fig. \ref{fig:ConstructionG+} (3)). We can see that $\mathcal{C}_{\max}(G^{(2)})=\{C_1,C_2,C_3,C_4\}$. It deduces that $G^{(i)}=G^{(2)}$ for all $i\geqslant 3$, i.e., $G^+=G^{(2)}$.
\end{exa}

About $G^{+}$, we have

\begin{lem}\label{lem:G+isMedianGraph}
For any partial cube $G$, $G^{+}$ is a median graph.
\end{lem}
\begin{proof}
First, we prove that $G^{+}$ is connected. In order to do this, we prove that $G^{(i)}$ is connected by induction on $i$. Since $G^{(0)}:=G$ is a partial cube, it is connected. Now, assume $G^{(i)}$ ($0\leqslant i\leqslant l-1$) is connected. By the definition of convex hulls, for each $C\in\mathcal{C}_{\max}(G^{(i)})$, $co_{Q_n}(C)$ is connected. In particular, for any $u\in co_{Q_n}(V(C))\setminus V(C)$, $v\in V(C)$, there exists a path between $u$ and $v$. Since $G^{(i)}$ is connected, it is deduced that $G^{(i+1)}$ is connected. Thus, $G^{+}(=G^{(l)})$ is connected by induction.

Next, we prove that the convex hull of any isometric cycle $C$ of $G^{+}(=G^{(l)})$ is a hypercube.

{\bf Case 1.} $C\in\mathcal{C}_{\max}(G^{(l)})$.

Since $S_{l+1}=S_l$, $co_{Q_n}(V(C))\subseteq S_l$. Then $co_{Q_n}(V(C))=co_{G^{(l)}}(V(C))$. By Observation \ref{obs:ConvexSubgraphofHypercube}, it induces a hypercube.

{\bf Case 2.} $C$ is not maximal.

By definition of `$\leqslant_{\mathcal{C}(G^{(l)})}$', there exists a maximal isometric cycle $C'$ such that $co_{G^{(l)}}(V(C))\subset co_{G^{(l)}}(V(C'))$. Similar to Case 1, $co_{G^{(l)}}(C')$ is a hypercube. By the definition of convexity, $co_{G^{(l)}}(C)$ is also a convex subgraph of $co_{G^{(l)}}(C')$.  By Observation \ref{obs:ConvexSubgraphofHypercube}, $co_{G^{(l)}}(C)$ is a hypercube.

In conclusion, the convex hull of every isometric cycle of $G^{+}(=G^{(l)})$ is a hypercube. By Theorem \ref{thm:ConvexHullofIsometricCycle}, $G^{+}$ is a median graph.
\end{proof}

Let $G$ be a partial cube and $H$ a subgraph of $G$. We denote $\mathcal{F}(H):=\{\theta|\theta\mbox{ is a }\Theta\mbox{-class of }G\mbox{ o-}$ $\mbox{ccurring in }H\}$. Now, we give a lemma.

\begin{lem}\label{lem:ConvexHullProperty}
Let $G$ be a partial cube and $H$ a subgraph of $G$. If $H$ is connected, then $\mathcal{F}(H)=\mathcal{F}(co_G(H))$.
\end{lem}
\begin{proof}
Since $H$ is the subgraph of $co_G(H)$, it is obvious that $\mathcal{F}(H)\subseteq\mathcal{F}(co_G(H))$. By contradiction assume $\mathcal{F}(H)\subset\mathcal{F}(co_G(H))$. For convenience, we denote $H':=co_G(H)$. Since $H'$ is an isometric subgraph of $G$, $H'$ is a partial cube by Observation \ref{obs:ThetaRestrictonSubgraph}. Assume $F_{uv}\in\mathcal{F}(H')\setminus\mathcal{F}(H)$. Since $H$ is connected, $H$ is a subgraph of $H'[W_{uv}^{H'}]$ or $H'[W_{vu}^{H'}]$. W.l.o.g., assume $H$ is the subgraph of $H'[W_{uv}^{H'}]$. Combined with Proposition \ref{pro:WisConvex} and the transitivity of convex subgraphs, $H'[W_{uv}^{H'}]$ is a convex subgraph of $G$, a contradiction with the fact that $H'$ is the smallest convex subgraph containing $H$.

Thus, $\mathcal{F}(H)=\mathcal{F}(co_G(H))$.
\end{proof}

Now, we consider the crossing graphs. Let $G$ be a partial cube. By the construction of $G^{(i)}$, $G^{(i-1)}$ is a subgraph of $G^{(i)}$ for any $1\leqslant i\leqslant l$. Moreover, for $u,v\in V(G^{(i-1)})$, we have
\begin{equation}\label{eq:Isometric}
d_{G^{(i-1)}}(u,v)=d_{G^{(i)}}(u,v).
\end{equation}
Thus, $G^{(i-1)}$ is an isometric subgraph of $G^{(i)}$. By Observation \ref{obs:ThetaRestrictonSubgraph} and Lemma \ref{lem:G+isMedianGraph}, $G^{(1)},G^{(2)},\cdots,$ $G^{(l-1)}$ are partial cubes. Set $\mathrm{idim}(G)=n$. By Lemma \ref{lem:ConvexHullProperty}, there are no new $\Theta$-classes arisen from $G^{(i-1)}$ to $G^{(i)}$ for any $1\leqslant i\leqslant l$. Then, we can deduce that $\mathrm{idim}(G^{(1)})=\mathrm{idim}(G^{(2)})=\cdots=\mathrm{idim}(G^{(l)})=n$. Moreover, for any $ab\in E(G^{(i-1)})$ ($1\leqslant i\leqslant l$), $F_{ab}^{G^{(i-1)}}\subseteq F_{ab}^{G^{(i)}}$ by Observation \ref{obs:ThetaRestrictonSubgraph}. For convenience, we can set that for any $ab\in E(G^{(i-1)})$, the corresponding vertex of $F_{ab}^{G^{(i-1)}}$ in $(G^{(i-1)})^{\#}$ and the one of $F_{ab}^{G^{(i)}}$ in $(G^{(i)})^{\#}$ are same. Thus, 
\begin{equation}\label{eq:VG}
V(G^{\#})=V((G^{(1)})^{\#})=\cdots=V((G^{(l)})^{\#}).
\end{equation}

\begin{lem}\label{lem:EqualCrossingGraph}
For any partial cube $G$ with $G\neq K_1$, $G^{\#}=(G^+)^{\#}$.
\end{lem}
\begin{proof}
We have known that $V(G^{\#})=V((G^+)^{\#})$ by (\ref{eq:VG}). Let $u,v$ be two vertices in $G^{\#}$, $ab,cd$ two edges in $E(G)$ such that $F_{ab}^G$ and $F_{cd}^G$ are the $\Theta$-classes of $G$ corresponding to $u$ and $v$ respectively.  Then $u$ (resp. $v$) is also the corresponding vertex of $F_{ab}^{G^{(i)}}$ (resp. $F_{cd}^{G^{(i)}}$) in $(G^{(i)})^{\#}$ for any $1\leqslant i\leqslant l$ by (\ref{eq:VG}). Now, we prove that

{\bf Claim 1.} $uv\in E(G^{\#})\Longrightarrow uv\in E((G^+)^{\#})$.

By Proposition \ref{pro:CrossinIsometricCycle}, $F_{ab}^G$, $F_{cd}^G$ occur on an isometric cycle $C$ in $G$. Combined with the fact that $G$ is an isometric subgraph of $G^+$ and Observation \ref{obs:ThetaRestrictonSubgraph}, $C$ is also an isometric cycle in $G^+$. Moreover, $F_{ab}^{G^+}$, $F_{cd}^{G^+}$ occur on $C$ in $G^+$. Thus, $uv\in E((G^+)^{\#})$ by Proposition \ref{pro:CrossinIsometricCycle}.

Then, we prove that

{\bf Claim 2.} $uv\not\in E(G^{\#})\Longrightarrow uv\not\in E((G^+)^{\#})$.

By contradiction assume $F_{ab}^{G^{+}}$ and $F_{cd}^{G^{+}}$ cross in $G^{+}$. By the construction of $G^+$, there exists an integer $i$ ($1\leqslant i\leqslant l$) such that $F_{ab}^{G^{(i)}}$ and $F_{cd}^{G^{(i)}}$ cross in $G^{(i)}$ but $F_{ab}^{G^{(i-1)}}$ and $F_{cd}^{G^{(i-1)}}$ don't cross in $G^{(i-1)}$. By the definition of crossing, w.l.o.g., assume $F_{ab}^{G^{(i-1)}}$ does not occur in $G^{(i-1)}[W_{cd}^{G^{(i-1)}}]$. However, since $F_{ab}^{G^{(i)}}$ and $F_{cd}^{G^{(i)}}$ cross in $G^{(i)}$, $F_{ab}^{G^{(i)}}$ occurs in $G^{(i)}[W_{cd}^{G^{(i)}}]$. Then there exists a maximal isometric cycle $C\in\mathcal{C}_{\max}(G^{(i-1)})$ such that $F_{ab}^{G^{(i)}}$ occurs in $co_{Q_n}(C)\cap G^{(i)}[W_{cd}^{G^{(i)}}]$. Since $F_{ab}^{G^{(i)}}\cap E(co_{Q_n}(C))\neq\emptyset$, by Lemma \ref{lem:ConvexHullProperty}, $F_{ab}^{G^{(i)}}$ occurs on $C$, so does $F_{ab}^{G^{(i-1)}}$.

Considering that $co_{Q_n}(V(C))\cap W_{cd}^{G^{(i)}}\neq\emptyset$, it can be deduced that
\begin{equation}\label{eq:CcapWcd}
V(C)\cap W_{cd}^{G^{(i-1)}}\neq\emptyset.
\end{equation}
Since $F_{ab}^{G^{(i-1)}}$ does not occur in $G^{(i-1)}[W_{cd}^{G^{(i-1)}}]$, all edges of $F_{ab}^{G^{(i-1)}}$ are in $E(G^{(i-1)}[W_{dc}^{G^{(i-1)}}])$. Combined with the fact that $F_{ab}^{G^{(i-1)}}$ occurs on $C$, we can obtain that
\begin{equation}\label{eq:CcapWdc}
V(C)\cap W_{dc}^{G^{(i-1)}}\neq\emptyset.
\end{equation}
Since $C$ is connected, $F_{cd}^{G^{(i-1)}}\cap E(C)\neq\emptyset$ by (\ref{eq:CcapWcd}) and (\ref{eq:CcapWdc}). That is,  both $F_{ab}^{G^{(i-1)}}$ and $F_{cd}^{G^{(i-1)}}$ occur on the isometric cycle $C$, a contradiction with Proposition \ref{pro:CrossinIsometricCycle}.

Combined with Claim 1, Claim 2 and (\ref{eq:VG}), we obtain that $G^{\#}=(G^+)^{\#}$.
\end{proof}

Now, we give the main result of this subsection.

\begin{thm}\label{thm:<forPartialCube}
Let $G$ be a partial cube and $G\neq K_1$. If $G$ is not a median graph, then
\begin{equation}\label{eq:<forPartialCube}
C(G,x)<Cl(G^{\#},x+1).
\end{equation}
\end{thm}
\begin{proof}
By the construction of $G^+$, $G$ is an induced subgraph of $G^+$, so every induced $i$-cube in $G$ is also an induced $i$-cube in $G^+$ for each $i\geqslant 0$. Thus, $\alpha_i(G)\leqslant\alpha_i(G^+)$. Since $G$ is not a median graph, $G\neq G^{+}$. Then $\alpha_0(G)=|V(G)|<|V(G^+)|=\alpha_0(G^+)$, and further $C(G,x)\neq C(G^+,x)$. Thus,
\begin{equation}\label{eq:CubePolynomialforSubgraph}
C(G,x)<C(G^+,x).
\end{equation}

Since $G^+$ is a median graph, by Theorem \ref{thm:=forMedianGraph},
\begin{equation}\label{eq:MedianGraph}
C(G^+,x)=Cl((G^+)^{\#},x+1).
\end{equation}

By Lemma \ref{lem:EqualCrossingGraph},
\begin{equation}\label{eq:EqualCrossingGraph}
Cl((G^+)^{\#},x+1)=Cl(G^{\#},x+1).
\end{equation}

Combined with (\ref{eq:CubePolynomialforSubgraph}), (\ref{eq:MedianGraph}) and (\ref{eq:EqualCrossingGraph}), we obtain that $C(G,x)<Cl(G^{\#},x+1)$.
\end{proof}

Combined with Theorems \ref{thm:=forMedianGraph} and \ref{thm:<forPartialCube}, Theorem \ref{thm:CubePandCliqueP} is proved.

\begin{comment}
Let's denote the set of all 0-1 sequences of length $n$ by $\mathcal{B}_n$, i.e., $V(Q_n)=\mathcal{B}_n$. Now, we introduce two binary Boolean operations on $\mathcal{B}_n$: `$\wedge$' and `$\vee$'. Let $(x_1x_2\cdots x_n)$ and $(y_1y_2\cdots y_n)$ be two 0-1 sequences of length $n$. $(z_1z_2\cdots z_n):=(x_1x_2\cdots x_n)\wedge(y_1y_2\cdots y_n)$ (resp. $(w_1w_2\cdots w_n):=(x_1x_2\cdots x_n)\vee(y_1y_2\cdots y_n)$) is defined as follows: If $x_i=y_i=1$ (resp. $x_i=y_i=0$), then $z_i=1$ (resp. $w_i=0$); otherwise $z_i=0$ (resp. $w_i=1$) for $1\leqslant i\leqslant n$. Let $v_1,v_2,\cdots,v_m$ be $m$ 0-1 sequences of length $n$, where $v_j=(x_1^jx_2^j\cdots x_n^j)$ for $1\leqslant j\leqslant m$. We know that the operations `$\wedge$' and `$\vee$' satisfy the commutative laws and the associative laws, so we can define $(z_1z_2\cdots z_n):=\bigwedge\limits_{j=1}^mv_j$ (resp. $(w_1w_2\cdots w_n):=\bigvee\limits_{j=1}^mv_j$): If $x^1_i=x^2_i=\cdots=x^m_i=1$ (resp. $x^1_i=x^2_i=\cdots=x^m_i=0$), then $z_i=1$ (resp. $w_i=0$); otherwise $z_i=0$ (resp. $w_i=1$) for any $i\in [n]$. For a subset $S$ of $\mathcal{B}_n$, we denote $\bigwedge S:=\bigwedge\limits_{v\in S}v$ (resp. $\bigvee S:=\bigvee\limits_{v\in S}v$).

We know that $\mathcal{B}_n$ forms a distributive lattice under `$\wedge$' and `$\vee$'.
\end{comment}
\section{Disproving Conjecture \ref{con:CubePsareUnimodal}}
Recall that a sequence $(s_1,s_2,\cdots,s_n)$ of nonnegative numbers is {\em unimodal} if
\begin{equation*}
s_1\leqslant s_2\leqslant\cdots\leqslant s_m\geqslant\cdots\geqslant s_{n-1}\geqslant s_n
\end{equation*}
for some integer $1\leqslant m\leqslant n$ and {\em log-concave} if
\begin{equation*}
s_{i-1}s_{i+1}\leqslant s^2_i,\qquad\mbox{for }2\leqslant i\leqslant n-1.
\end{equation*}

The clique polynomial of a graph is not necessarily unimodal, as shown in the following example:

\begin{exa}\label{exa:NonUnimodalCliqueP}
Let $n,m$ be nonnegative integers with $n\geqslant 6$, $G:=K_n\cup mK_1$ the graph of the disjoint union of a complete graph $K_n$ and $m$ single vertices. Then
\begin{equation*}
Cl(G,x)=(x+1)^n+mx.
\end{equation*}
After calculations, we obtain that $Cl(G,x)$ is log-concave when $m\leqslant\left\lfloor\frac{n^2+n}{2n-4}\right\rfloor$; $Cl(G,x)$ is unimodal but not log-concave when $\left\lfloor\frac{n^2+n}{2n-4}\right\rfloor+1\leqslant m\leqslant\frac{n^2-3n}{2}$; $Cl(G,x)$ is not unimodal when $m\geqslant\frac{n^2-3n}{2}+1$.
\end{exa}

Combined with Theorems \ref{thm:=forMedianGraph} and \ref{thm:EveryGraphisaCrossingGraph}, we can construct median graphs with non-unimodal cube polynomials, as shown in the following example:

\begin{exa}\label{exa:NonUnimodalCubeP}
Let $n,m$ be nonnegative integers with $n\geqslant 9$, $G$ the graph formed by $n$-cube $Q_n$ with  $m$ pendant vertices attached. We can obtain that $G$ is a median graph and $G^{\#}\cong K_n\cup mK_1$. Then
\begin{equation*}
C(G,x)=Cl(K_n\cup mK_1,x+1)=(x+2)^n+m(x+1).
\end{equation*}
After calculations, we obtain that $C(G,x)$ is log-concave when $m\leqslant\left\lfloor\frac{n^2+n}{n-2}\cdot 2^{n-2}\right\rfloor$;$C(G,x)$ is unimodal but not log-concave when $\left\lfloor\frac{n^2+n}{n-2}\cdot 2^{n-2}\right\rfloor+1\leqslant m\leqslant\frac{n^2-5n}{2}\cdot 2^{n-2}$; $C(G,x)$ is not unimodal when $m\geqslant\frac{n^2-5n}{2}\cdot 2^{n-2}+1$.
\end{exa}

More generally, let $G$ be any median graph satisfying $\alpha_3(G)>\alpha_2(G)$ (the above $Q_n$ $(n\geqslant 9)$ as examples). Whether $C(G,x)$ is unimodal or not, the cube polynomial of the graph obtained from $G$ by attaching sufficiently many ( $\geqslant \alpha_2(G)-\alpha_1(G)+1$) pendant vertices  is not unimodal.

Thus, Conjecture \ref{con:CubePsareUnimodal} is false.

\section{Conclusion and future work regarding  unimodality}
In the present paper, we obtain an inequality relation between the cube polynomials of partial cubes $G$ and clique polynomials of their crossing graphs, i.e.,  $C(G,x)\leqslant Cl(G^{\#},x+1)$. Moreover, the equality holds if and only if $G$ is a median graph.

A {\em hexagonal system} (or {\em benzenoid system}) is a 2-connected finite plane graph such that every interior face is a regular hexagon of side length one. Let $H$ be a hexagonal system. The Clar covering polynomial (or Zhang-Zhang polynomial) $\zeta(H, x)$ is an important graph polynomial in mathematical chemistry, which was introduced by Zhang and Zhang \cite{zz96}. Zhang et al. \cite{zss13} proved that $\zeta(H,x)=C(R(H),x)$, where $R(H)$ is the resonance graph of $H$ (further a median graph \cite{zls08}).

Although Conjecture \ref{con:CubePsareUnimodal} associated with median graphs is false, the conjecture on the resonance graphs of hexagonal systems (a subclass of median graphs) is still open.

\begin{con}{\em\cite{zz00}}\label{con:ClarPisUnimodal}
For a hexagonal system $H$, $C(R(H),x)$ (i.e., $\zeta(H,x)$) is unimodal.
\end{con}

Further, since all coefficients of $\zeta(H,x)$ are positive, a stronger conjecture was proposed by Li et al. \cite{lpw20} after checking large amounts of numerical results as follows.

\begin{con}{\em\cite{lpw20}}\label{con:ClarPisLogConcave}
For a hexagonal system $H$, the Clar covering polynomial $\zeta(H,x)$ is log-concave.
\end{con}

By Theorems \ref{thm:=forMedianGraph} and \ref{thm:EveryGraphisaCrossingGraph}, the cube polynomials of median graphs, the clique polynomials and the independence polynomials of general graphs can be transformed into each other. Considering the close relationship among these polynomials, here we mentioned the following well-known conjecture proposed by Alavi, Malde, Schwenk and Erd\H{o}s in 1987.

\begin{con}{\em\cite{amse87}}\label{con:IndependencePofTreeisUnimodal}
The independence polynomial of every tree is unimodal.
\end{con}

In particular, Conjecture \ref{con:IndependencePofTreeisUnimodal} cannot be strengthened up to its log-concave version, which is disproved by Kadrawi et al. \cite{kl23,klym23} by providing counterexamples---several infinite families of trees of order at least 26.

For some special unimodal polynomials $P(x)$, e.g. log-concave polynomials \cite{h74}, nondecreasing polynomials \cite{bm99} or unimodal polynomials with degree $n$ and mode at least $n-4$ \cite{xx17}, $P(x+1)$ is also unimodal. Combined with Theorem \ref{thm:=forMedianGraph}, the following problem related to Conjecture \ref{con:IndependencePofTreeisUnimodal} can be studied in the future. Note that a {\em co-tree} is the complement of a tree.

\begin{prob}\label{prob:CubePareUnimodal}
Is the cube polynomial $C(G,x)$ unimodal if $G^{\#}$ is a co-tree?
\end{prob}
\begin{comment}
The following conjecture is stronger.

\begin{con}{\em\cite{bg21}}\label{con:IndependencePofTreeisLC}
The independence polynomial of every tree is log-concave.
\end{con}

About log-concavity, the following proposition is well-known.

\begin{pro}{\em\cite{h74}}\label{pro:P(x+1)isLC}
If the polynomial $P(x)$ with positive coefficients is log-concave, then so does $P(x+1)$.
\end{pro}

A graph is called a {\em co-tree} if it is the complement of a tree. Combined with Theorem \ref{thm:=forMedianGraph} and Proposition \ref{pro:P(x+1)isLC}, we give the following conjecture, which is weaker than Conjecture \ref{con:IndependencePofTreeisLC}.

\begin{con}\label{con:CubePsareLC}
Let $G$ be a median graph. The cube polynomial $C(G,x)$ is log-concave if $G^{\#}$ is a co-tree.
\end{con}
Thus, the structural properties of median graphs whose crossing graphs are co-trees are worthy for future studies.
\end{comment}

\vskip 0.2 cm
\noindent{\bf Acknowledgements:} The authors thank the referees for their careful reviews. This work is partially supported by National Natural Science Foundation of China (Grants No. 12071194, 11571155, 11961067).


\begin{thebibliography}{99}
\parskip -0.1cm

\bibitem{amse87}Y. Alavi, P. J. Malde, A. J. Schwenk, P. Erd\"os, The vertex independence sequence of a graph is not constrained, Congressus Numerantium, 58 (1987) 15--23.

%\bibitem{b82}H.-J. Bandelt, Characterizing median graphs, 1982, manuscript.

\bibitem{bd92}H.-J. Bandelt, A. W. M. Dress, A canonical decomposition theory for metrics on a finite set, Adv. Math., 92 (1992) 47--105.

\bibitem{bv87}H.-J. Bandelt, M. van de Vel, A fixed cube theorem for median graphs, Discrete Math., 62 (1987) 129--137.

\bibitem{bv91}H.-J. Bandelt, M. van de Vel, Superextensions and the depth of median graphs, J. Combin. Theory Ser. A, 57 (1991) 187--202.

%\bibitem{bg21}A. Basit, D. Galvin, On the independent set sequence of a tree, Electron. J. Combin., 28 (2021) (3), Paper No. 3.23.

\bibitem{bm99}G. Boros, V. H. Moll, A criterion for unimodality, Electron. J. Combin., 6 (1999), \#R10.

\bibitem{b89}F. Brenti, Unimodal, log-concave, and P\'olya frequency sequences in combinatorics, Mem. Amer. Math. Soc., 81 (1989) 413.

%\bibitem{bik03}B. Bre\v sar,  W. Imrich, S. Klav\v zar, Tree-like isometric subgraphs of hypercubes, Discuss. Math. Graph Theory, 23 (2003) 227--240.

\bibitem{bks03}B. Bre\v sar, S. Klav\v zar, R. \v Skrekovski, The cube polynomial and its derivatives: the case of median graphs, Electron. J. Combin., 10 (2003) \#R3.

%\bibitem{bks06}B. Bre\v sar, S. Klav\v zar, R. \v Skrekovski, Roots of cube polynomials of median graphs, J. Graph Theory 52 (2006) 37–50.

\bibitem{c88}V. Chepoi, $d$-Convexity and isometric subgraphs of Hamming graphs, Cybernetics, 1 (1988) 6--10.

\bibitem{dj73}D. \v{Z}. Djokovi\'c, Distance preserving subgraphs of hypercubes, J. Combin. Theory Ser. B, 14 (1973) 263--267.

%\bibitem{g82}W. Gr\"undler, Signifikante Elektronenstrukturen fur benzenoide Kohlenwasserstoffe, Wiss. Z. Univ. Halle 31 (1982) 97--116.

\bibitem{h90}Y. O. Hamidoune, On the numbers of independent $k$-sets in a clawfree graph, J. Combin. Theory Ser. B, 50 (1990) 241--244.

\bibitem{hik11}R. Hammack, W. Imrich, S. Klav\v{z}ar, Handbook of product graphs, Boca Raton, CRC press, 2011.

\bibitem{hl72}O. J. Heilmann, E. H. Lieb, Theory of monomer–dimer systems, Comm. Math. Phys., 25 (1972) 190--232.

\bibitem{hl94}C. Hoede, X. Li, Clique polynomials and independent set polynomials of graphs, Discrete Math., 125 (1994) 219--228.

\bibitem{h74}S. G. Hoggar, Chromatic polynomials and logarithmic concavity, J. Combin. Theory Ser. B, 16 (1974), 248--254.

\bibitem{h12}J. Huh, Milnor numbers of projective hypersurfaces and the chromatic polynomial of graphs, J. Amer. Math. Soc., 25 (3) (2012) 907--927.

\bibitem{ik98}W. Imrich, S. Klav\v zar, A convexity lemma and expansion procedures for bipartite graphs, European J. Combin., 19 (1998) 677--685.

\bibitem{ik00}W. Imrich, S. Klav\v zar, Product Graphs: Structure and Recognition, John Wiley \& Sons, New York, USA, 2000.

\bibitem{kl23}O. Kadrawi, V. E. Levit, The independence polynomial of trees is not always log-concave starting from order 26, arXiv preprint arXiv: 2305.01784, 2023.

\bibitem{klym23}O. Kadrawi, V. E. Levit, R. Yosef, M. Mizrachi, On computing of independence polynomials of trees, F. \"Ozger (Ed.), Recent Research in Polynomials, IntechOpen 2023. DOI:10.5772/intechopen.1001130.

\bibitem{km99}S. Klav\v zar, H. M. Mulder, Median graphs: characterizations, location theory and related structures, J. Combin. Math. Combin. Comput., 30 (1999) 103--127.

\bibitem{km02}S. Klav\v zar, H. M. Mulder, Partial cubes and crossing graphs, SIAM J. Discrete Math., 15 (2002) 235--251.

\bibitem{lpw20}G. Li, Y. Pei, Y. Wang, Clar covering polynomials with only real zeros, MATCH Commun. Math. Comput. Chem., 84 (1) (2020) 217--228.

\bibitem{mu78}H. M. Mulder, The structure of median graphs, Discrete Math., 24 (1978) 197--204.

\bibitem{mu80a}H. M. Mulder, $n$-cubes and median graphs, J. Graph Theory, 4 (1980) 107--110.

\bibitem{mu80b}H. M. Mulder, The interval function of a graph, PhD thesis, Vrije Universiteit Amsterdam, 1980.

\bibitem{mu90}H. M. Mulder, The expansion procedure for graphs, in: R. Bodendiek (Ed.), Contemporary Methods in Graph Theory, Wissenschaftsverlag, Mannheim, 1990, pp. 459--477.

\bibitem{mu11}H. M. Mulder, Median graphs. A structure theory, in: H. Kaul, H.M. Mulder (Eds.), Advances in Interdisciplinary Applied Discrete Mathematics, vol. 11, World Scientific, Singapore, 2011, pp. 93--125.

%\bibitem{r97}M. Randi\'c, Resonance in catacondensed benzenoid hydrocarbons, Int. J. Quantum Chem., 63 (1997) 585--600.

%\bibitem{r03}M. Randi\'c, Aromaticity of polycyclic conjugated hydrocarbons, Chem. Rev., 103 (2003) 3449--3605.

\bibitem{w84}P. M. Winkler, Isometric embedding in products of complete graphs, Discrete Appl. Math., 7 (1984) 221--225.

\bibitem{xx17}Y.-T. Xie, S.-J. Xu, Nested unimodality, Australas. J. Combin., 69 (1) (2017) 119--129.

%\bibitem{zgc88a}F. Zhang, X. Guo, R. Chen, Z-transformation graph of perfect matchings of hexagonal systems, Discrete Math., 72 (1988) 405--415.

%\bibitem{zgc88b}F. Zhang, X. Guo, R. Chen, The connectivity of Z-transformation graphs of perfect matchings of hexagonal systems, Acta Math. Appl. Sinica, 4 (1988) 131--135.

\bibitem{zls08}H. Zhang, P. C. B. Lam, W. C. Shiu, Resonance graphs and a binary coding for the 1-factors of benzenoid systems, SIAM J. Discrete Math., 22 (2008) 971--984.

\bibitem{zss13}H. Zhang, W. C. Shiu, P. K. Sun, A relation between Clar covering polynomial and cube polynomial, MATCH Commun. Math. Comput. Chem. 70 (2013) 477--492.

\bibitem{zz96}H. Zhang, F. Zhang, The Clar covering polynomial of hexagonal systems I, Discrete Appl. Math., 69 (1996) 147--167.

\bibitem{zz00}H. Zhang, F. Zhang, The Clar covering polynomial of hexagonal systems III, Discrete Math., 212 (2000) 261--269.
\end{thebibliography}
\end{document}